\documentclass[psamsfont,a4paper,12pt]{amsart}

\usepackage[english]{babel}   

\usepackage{setspace}     
\usepackage{enumerate} 	
\usepackage{afterpage} 	
\usepackage[usenames,dvipsnames]{color}
\usepackage{graphicx}	

\usepackage{amsmath} 	
\usepackage{amssymb}	
\usepackage{mathrsfs} 	
\usepackage{amsfonts} 	
\usepackage{latexsym} 	
\usepackage[latin1]{inputenc} 

\usepackage{amsthm} 	
\usepackage{stackrel}       
\usepackage{amscd} 	

\usepackage{tabularx}	
\usepackage{longtable}	

\usepackage[vcentermath]{youngtab}    

\usepackage{float}

\usepackage{verbatim} 
\usepackage{blindtext} 

\usepackage{lipsum}

\allowdisplaybreaks

\def\cal{\mathcal}

\def\L{{\cal L}}

\newcommand{\C}{\mathbb{C}}

\newcommand{\R}{\mathbb{R}}

\newcommand{\NN}{\mathbb{N}}
\newcommand{\Z}{\mathbb{Z}}

\renewcommand{\to}{\longrightarrow}

\newtheorem{Thm}{Theorem}[section]		
\newtheorem{Lemma}[Thm]{Lemma}
\newtheorem{Conj}[Thm]{Conjecture}

\theoremstyle{definition}

\newtheorem*{defn}{Definition}

\theoremstyle{remark}

\newtheorem*{rmk}{Remark}

\newtheorem{ind}[]{{\rm\it Indice}}

\usepackage{multicol}

\usepackage{tikz}
\usetikzlibrary{arrows}

\usepackage{paralist}
\usepackage{cite}

\usepackage{breqn}   

\title[On a new class of Laguerre-P\'{o}lya type functions]{On a new class of Laguerre-P\'{o}lya type functions with applications in number theory}
\author[Wagner]{Ian Wagner}
\email{iwagne11@uni-koeln.de}

\usepackage{listings}

\begin{document}
\numberwithin{equation}{section}

\maketitle

\begin{abstract}
We define a new class of functions, connected to the classical Laguerre-P\'{o}lya class, which we call the shifted Laguerre-P\'{o}lya class.  Recent work of Griffin, Ono, Rolen, and Zagier shows that the Riemann Xi function is in this class.  We prove that a function being in this class is equivalent to its Taylor coefficients, once shifted, being a degree $d$ multiplier sequence for every $d$, which is equivalent to its shifted coefficients satisfying all of the higher Tur\'{a}n inequalities.  This mirrors a classical result of P\'{o}lya and Schur.  For each function in this class we show some order derivative satisfies each extended Laguerre inequality.  Finally, we discuss some old and new conjectures about iterated inequalities for functions in this class. 
\end{abstract}

\section{Introduction and results}
A partition of a non-negative integer $n$ is a non-increasing of positive integers that sum to $n$.  The partition function $p(n)$ counts the number of partitions of $n$.  In the 1970's Nicolas \cite{N} used asymptotics to prove that $p(n)$ is log-concave for $n>25$, that is
\begin{equation*}
p(n)^2 \geq p(n-1) p(n+1), \qquad n>25.
\end{equation*}
In \cite{RS1} Stanley surveyed several surprising results on the log-concavity of sequences arising from algebra, combinatorics, and geometry.  There has been a recent resurgence of studying the log-concavity or eventual log-concavity of sequences along the lines of Stanley.  In particular, in \cite{DP} the log-concavity of the partition function was reproved, which inspired Chen, Jia, and Wang \cite{CJW} to study higher order analogues of these inequalities which in turn has led to numerous generalizations and results pertaining to inequalities for combinatorial sequences (see for example \cite{HN}, \cite{DM}, \cite{JW}).  The aim of this work is to put these results into a general framework and, in some instances, give extensions of the discussed results.

A real entire function is said to be in the Laguerre-P\'{o}lya class if it is the uniform limit on compact subsets of $\C$ of a sequence of real polynomials with all real roots.  Further, a function in the Laguerre-P\'{o}lya class is said to be of type I if it is the limit of a sequence of polynomials with all real and non-positive roots \cite{PS}.  Equivalently, functions in this class can be given by the following definition.
\begin{defn}
A real entire function $\psi(x) := \sum_{k \geq 0} \frac{\gamma_{k}}{k!} x^k$ is said to belong to the {\bf{Laguerre-P\'{o}lya class}}, denoted by $\mathcal{L\text{-}P}$, if it can be represented as
\begin{equation*}
\psi(x) = Cx^m e^{bx-ax^2} \prod_{k=1}^{r} \left(1 + \frac{x}{x_{k}} \right)e^{-\frac{x}{x_{k}}}, \qquad 0 \leq r \leq \infty
\end{equation*}
where $b, C, x_{k} \in \R$, $m \in \Z_{\geq 0}$, $a \geq 0$, and $\sum_{k=1}^{r} x_{k}^{-2} < \infty$.  If, for a real entire function $\psi(x) \in \mathcal{L\text{-}P}$, either $\psi(x)$ or $\psi(-x)$ can be represented in the form
\begin{equation*}
\psi(x) = C x^m e^{\sigma x} \prod_{k=1}^{r} \left( 1 + \frac{x}{x_{k}} \right), \qquad 0 \leq r \leq \infty
\end{equation*}
with $C \in \R$, $m \in \Z_{\geq 0}$,  $\sigma \geq 0$, $x_{k} >0$, and $\sum x_{k}^{-1} < \infty$, then we say $\psi \in \mathcal{L\text{-}P}I$ is of {\bf{type I}} in the Laguerre-P\'{o}lya class.  Finally, if $\gamma_{k} \geq 0$ for all $k \geq 0$ for a function $\psi \in \mathcal{L\text{-}P}I$ then we say $\psi \in \mathcal{L\text{-}P}^{+}$.
\end{defn}
In this work we will focus on a generalization of the Laguerre-P\'{o}lya class of type I, but all of the results can be adjusted to fit the full Laguerre-P\'{o}lya class.  In \cite{P} P\'{o}lya proved that the Riemann hypothesis is equivalent to the Riemann Xi-function
\begin{equation*}
\Xi(z) := \frac{1}{2} \left(-z^2 - \frac{1}{4} \right) \pi^{\frac{iz}{2} - \frac{1}{4}} \Gamma \left(- \frac{iz}{2} + \frac{1}{4} \right) \zeta \left( -iz + \frac{1}{2} \right)
\end{equation*}
being in the Laguerre-P\'{o}lya class.

We will now introduce two topics related to this class of functions.  Given a sequence of real numbers $\{\gamma_{k} \}_{k \geq 0}$ we can define the linear operator $\Gamma_{\gamma} \in \L(\R[[x]])$ by $\Gamma_{\gamma}(x^k) = \gamma_{k} x^k$.
\begin{defn}
A sequence of real numbers $\{\gamma_{k} \}_{k \geq 0}$ is called a {\bf{multiplier sequence of type I}} if $\Gamma_{\gamma}(p(x))$ has only real zeros whenever the real polynomial $p(x)$ has only real zeros and is a {\bf{multiplier sequence of type II}} if $\Gamma_{\gamma}(p(x))$ has only real zeros whenever $p(x)$ has only real zeros of the same sign.
\end{defn}
To avoid any confusion we can extend any finite sequence $\{a_{k}\}_{k=0}^m$ to an infinite one $\{a_k\}_{k \geq 0}$ by just defining $a_{k} =0$ for $k>m$.  Multiplier sequences satisfy several nice properties.  In particular, if $\{\gamma_{k}\}_{k \geq 0}$ is a multiplier sequence, then so is $\{\gamma_{k+n}\}_{k \geq 0}$ for any nonnegative integer $n$, and with the Hadamard product as the operation the set of all multiplier sequences forms a monoid.  That is, if $\{\gamma_{k} \}_{k \geq 0}$ and $\{\lambda_k \}_{k \geq 0}$ are multiplier sequences, then $\{\gamma_{k} \}_{k \geq 0} * \{\lambda_k \}_{k \geq 0} = \{\gamma_{k} \lambda_{k} \}_{k \geq 0}$ is as well.  Given a sequence of real numbers $\gamma = \{\gamma_{k}\}_{k \geq 0}$ we define the {\bf{Jensen polynomial of degree $d$ and shift $n$ of $\gamma$}} by   
\begin{equation} \label{J}
J_{\gamma}^{d,n}(x) := \Gamma_{\{\gamma_{k+n}\}}((1+x)^d) = \sum_{k=0}^{d} \binom{d}{k} \gamma_{k+n} x^k.
\end{equation}

In \cite{PS} P\'{o}lya and Schur showed that if $\{\gamma_{k}\}_{k \geq 0}$ is a sequence of nonnegative real numbers, then the following are equivalent:
\begin{enumerate}
\item $\{\gamma_{k} \}_{k \geq 0}$ is a multiplier sequence.
\item For each $d$, the polynomial $J_{\gamma}^{d,0}(x)$ has all real non-positive roots; that is $J_{\gamma}^{d,0}(x) \in \mathcal{L\text{-}P}I$.
\item The formal power series $\phi(x) = \sum_{k=0}^{\infty} \frac{\gamma_{k}}{k!}x^k$ defines a function in the Laguerre-P\'{o}lya class of type I; that is $\Gamma_{\gamma}(e^x) \in \mathcal{L\text{-}P}I$.
\end{enumerate}

Note that we can essentially just work with $\psi \in \mathcal{L\text{-}P}^{+}$ since $\{-1\}_{k \geq 0}$, $\{(-1)^k \}_{k \geq 0}$, and $\{(-1)^{k+1} \}_{k \geq 0}$ are all multiplier sequences and it is known that for a multiplier sequence $\{\gamma_{k} \}_{k \geq 0}$ we have ${\rm{sign}}(\gamma_{k-1}) = {\rm{sign}}(\gamma_{k+1})$ \cite{PS}.

Recently, Griffin, Ono, Rolen, and Zagier \cite{GORZ} proved that if a sequence of real numbers $\{\gamma_{k}\}_{k \geq 0}$ satisfy certain conditions, then $\lim_{n \to \infty} \widetilde{J}_{\gamma}^{d,n}(x) = H_{d}(x)$, where $\widetilde{J}_{\gamma}^{d,n}(x)$ is a certain renormalization of the Jensen polynomials and $H_{d}(x)$ is the degree $d$ Hermite polynomial.  In particular, this shows that for these sequences, for each $d$ there exists an $N(d)$ such that $J_{\gamma}^{d,n}(x)$ has all real roots for each $n \geq N(d)$.  In \cite{GORTTW} Griffin, Ono, Rolen, Thorner, Tripp, and the author made this result effective for the Riemann Xi function.  It is natural to ask how such sequences fit into the theory of multiplier sequences and the Laguerre-P\'{o}lya class.  With this as motivation we make the following definitions, where for a polynomial $\phi_{n}(x)$ we denote the degree by $d_{n} = \deg(\phi_n)$.

\begin{defn}
A real entire function $\phi(x)$ is said to belong to the {\bf{shifted Laguerre-P\'{o}lya class of degree $d$}}, denoted $\mathcal{SL\text{-}P}(d)$, if it is the uniform limit of polynomials $\{\phi_{k}\}_{k \geq 0}$ with the property that there exists an $N(d)$ such that $\phi_{n}^{(d_n -d)}(x)$ has all real roots for any $n \geq N(d)$.  We say $\phi \in \mathcal{SL\text{-}P}(d)$ is of {\bf{type I}} and write $\phi \in \mathcal{SL\text{-}P}I(d)$ if all of the roots of $\phi_{n}^{(d_n -d)}(x)$ have the same sign.

A real entire function $\phi(x) = \sum_{k \geq 0} \frac{\gamma_{k}}{k!}x^k$ is said to belong to the {\bf{shifted Laguerre-P\'{o}lya class}}, denoted by $\mathcal{SL\text{-}P}$, if $\phi \in \mathcal{SL\text{-}P}(d)$ for every $d \in \NN$.  That is, it is the uniform limit on compact subsets of $\C$ of a sequence of real polynomials $\{\phi_{k}(x)\}_{k \geq 0}$ such that for each $d \in \NN$, there exists an $N(d)$ such that $\phi_{n}^{(d_n-d)}(x)$ has all real roots for any $n \geq N(d)$.  This implies that $\phi(x)$ has order at most $2$.  
If, for an entire real function $\phi \in \mathcal{SL\text{-}P}$ the sequence of polynomials has the property that $\phi_{n}^{(d_n -d)}(x)$ has all real roots of the same sign for any $n \geq N(d)$ then we say it is of {\bf type I} and write $\phi \in \mathcal{SL\text{-}P}I$.  This implies that $\phi(x)$ has order at most $1$.  Finally, if $\gamma_{k} \geq 0$ for $k$ large enough for a function $\phi \in \mathcal{SL\text{-}P}I$, then we write $\phi \in \mathcal{SL\text{-}P}^{+}$.
\end{defn}
\begin{rmk}
It is clear that $\mathcal{SL\text{-}P}(d) \subset \mathcal{SL\text{-}P}(d-1)$.  Further, if $\phi \in \mathcal{L\text{-}P}$, then $\phi \in \mathcal{SL\text{-}P}(d)$ for all $d \leq \deg(\phi)$ (or $d$ is any nonnegative integer if $\phi$ is transcendental) and one can take $N(d) = 0$ so this should be viewed as the shift zero case.
\end{rmk}
\begin{rmk}
It is shown in the proof of Theorem \ref{SLP} that function $\phi(x) \in \mathcal{SL\text{-}P}I$ has Weierstrass factorization
\begin{equation*} 
\phi(x) = C x^m e^{\sigma x} \prod_{k=1}^{r} \left( 1 + \frac{x}{x_{k}} \right), \qquad 0 \leq r \leq \infty
\end{equation*}
with $C \in \R, m \in \Z_{\geq 0}, \sigma \geq 0$, and $\sum |x_k|^{-1} < \infty$
or
\begin{equation*}
\phi(x) = C x^m e^{\sigma x} \prod_{k=1}^{r} \left( 1 + \frac{x}{x_{k}} \right)e^{-\frac{x}{x_k}}, \qquad 0 \leq r \leq \infty
\end{equation*}
with $\sum|x_{k}|^{-2} < \infty$.
\end{rmk}
\begin{rmk}
It is clear from the definition that the Laguerre-P\'{o}lya class is closed under differentiation.  Further, the P\'{o}lya-Wiman conjecture, proved by Craven, Csordas, and Smith \cite{CCS}, states that if $\phi(x)$ is a real entire function of order less than $2$ with a finite number of non-real zeros, then there exists an $N$ such that $\phi^{(m)}(x) \in \mathcal{L\text{-}P}$ for all $m \geq N$.  Kim \cite{K} extended this result to all $\phi \in \mathcal{L\text{-}P}^{*}$, where $\mathcal{L\text{-}P}^{*}$ is the class of all $\phi(x)= P(x)h(x)$ where $P$ is a real polynomial and $h \in \mathcal{L\text{-}P}$.  For such a function this implies $J_{\gamma}^{d,m}(x)$ has all real roots and so by Lemma \ref{converge} below one can take the associated Jensen polynomials as the sequence of polynomials to converge to $\phi(x)$.  In this context we have the following inclusions for transcendental functions
\begin{equation*}
\mathcal{L\text{-}P} \subseteq \mathcal{L\text{-}P}^* \subseteq \mathcal{SL\text{-}P}.
\end{equation*}
\end{rmk}

\begin{defn}
Given a nonnegative integer $d$, a sequence of real numbers $\{\gamma_{k}\}_{k \geq 0}$ is called an {\bf{order $d$ multiplier sequence of type I}} if, for each $n \in \NN$, $\Gamma_{\{\gamma_{k+n}\}}(p(x))$ has only real zeros whenever $p(x)$ has only real zeros and $\deg(p(x)) \leq d$ and an {\bf{order $d$ multiplier sequence of type II}} if, for each $n \in \NN$, $\Gamma_{\{\gamma_{k+n}\}}(p(x))$ has only real zeros whenever $p(x)$ has only real zeros of the same sign and $\deg(p(x)) \leq d$.

A sequence of real numbers $\{\gamma_{k}\}_{k \geq 0}$ is called a {\bf{shifted multiplier sequence of type I}}  (resp. {\bf{type II)}} if for each $d \in \NN$, there exists an $N(d)$ such that $\{\gamma_{k+n}\}_{k \geq 0}$ is an order $d$ multiplier sequence of type I  (resp. type II) for all $n \geq N(d)$.
\end{defn}
We can now state our main theorem.
\begin{Thm} \label{SLP}
Let $\{\gamma_{k}\}_{k \geq 0}$ be a sequence of eventually nonnegative real numbers.  Then the following are equivalent.
\begin{enumerate}
\item $\{\gamma_{k}\}_{k \geq 0}$ is a shifted multiplier sequence of type I.
\item For each $d \in \NN$, there exists an $N_{2}(d)$ such that $J_{\gamma}^{d,n}(x)$ has all real non-positive roots for all $n \geq N_{2}(d)$; that is $J_{\gamma}^{d+n,0}(x) \in \mathcal{SL\text{-}P}I(d)$ for all $n \geq N_{2}(d)$.
\item The formal power series $\phi(x) = \sum_{k =0}^{\infty} \frac{\gamma_{k}}{k!} x^k$ defines a function in the shifted Laguerre-P\'{o}lya class of type I; that is $\Gamma_{\gamma}(e^x) \in \mathcal{SL\text{-}P}I$.
\end{enumerate}
\end{Thm}
\begin{rmk}
The general philosophy of functions in the shifted Laguerre-P\'{o}lya class is that the Taylor coefficients $\gamma_{k}$ should behave more and more like Taylor coefficients of a function in the Laguerre-P\'{o}lya class as $k$ grows.
\end{rmk}
\begin{rmk}
The above theorem can be adjusted so the results hold for order $d$ multiplier sequences of the second kind and functions in $\mathcal{SL\text{-}P}$.
\end{rmk}
\begin{rmk}
A recent result of Ono, Pujahari, and Rolen \cite{OPR} shows that
\begin{equation*}
\sum_{k=0}^{\infty} \frac{pp(k)}{k!} x^k \in \mathcal{SL\text{-}P}I,
\end{equation*}
where $pp(k)$ is the number of plane partitions of $k$.  This is unique amongst recent results as the generating function for $pp(k)$ is not modular.  In \cite{HN2} evidence is given that similar behavior may hold for many sequences with infinite product generating functions.
\end{rmk}

For any real entire function $\phi(x)$, one can define the operators $\{L_{k}\}_{k \geq 0}$ by
\begin{equation} \label{L}
|\phi(x+iy)|^2 = \phi(x+iy)\phi(x-iy) = \sum_{k=0}^{\infty} L_{k}(\phi(x)) y^{2k}, \qquad x, y \in \R.
\end{equation}
The first few operators are $L_{0}(\phi(x)) = \phi(x)^2$ and $L_{1}(\phi(x)) = \phi'(x)^2 - \phi(x) \phi''(x)$.  $L_{1}(\phi(x))$ was first studied by Laguerre and can clearly bee seen to be related to the zeros of $\phi(x)$ by $L_{1}(\phi(x)) = -\phi(x)^2 \frac{\partial^2}{\partial x^2} \ln(\phi(x))$.  Using the generalized Leibniz rule the other operators can be given explicitly by
\begin{equation} \label{L2}
L_{k}(\phi(x)) = \sum_{j=0}^{2k} \frac{(-1)^{j+k}}{(2k)!} \binom{2k}{j} \phi^{(j)}(x) \phi^{(2k-j)}(x).
\end{equation}
Csordas and Varga \cite{CV1} studied these operators and later Csordas and Vishnyakova \cite{CV2} proved that $\phi \in \mathcal{L\text{-}P}$ is equivalent to $L_{k}(\phi(x)) \geq 0$ for all $k \in \NN$ and all $x \in \R$.  We are able to prove one direction of the analogue of this result and conjecture that the other is also true.
\begin{Thm} \label{LThm}
Let $\phi(x) \in \mathcal{SL\text{-}P}$.  Then for each $d \in \NN$, there exists an $N_{3}(d)$ such that $L_{k}(\phi^{(n)}(x)) \geq 0$ for all $0 \leq k \leq d, n \geq N_{3}(d)$ and $x \in \R$.
\end{Thm}

\begin{Conj}
The converse direction of Theorem \ref{LThm} is true.
\end{Conj}

The results of \cite{GORZ} then say that $\Xi(i \sqrt{x}) \in \mathcal{SL\text{-}P}$.  In \cite{W} these results were extended to include any suitably nice completed $L$-function $\Xi_{L}(x)$, where suitably nice means the $L$-function has real coefficients, an integral representation, and a functional equation.  Further, if $\{\gamma_{k}\}_{k \geq 0}$ are the Taylor coefficients of a modular form (with possibly fractional weight and with multiplier system) on $SL_{2}(\Z)$ with real Fourier coefficients and holomorphic apart from a pole of positive order at infinity, then it is also shown in \cite{GORZ} that $\phi_{\gamma}(x) := \sum_{k=0}^{\infty}\frac{\gamma_k}{k!}x^k \in \mathcal{SL\text{-}P}I$.   As an immediate application of Theorem \ref{SLP} and Theorem \ref{LThm} we have the following result.
\begin{Thm} \label{PL}
Assume the notation above and let $\gamma = \{\gamma_{k}\}_{k \geq 0}$ be the coefficients of a suitable modular form or the Taylor coefficients of a suitable completed $L$-function $\Xi_{L} \left(x \right)$.  For each $d \in \NN$, there exists an $N_{4}(d)$ such that
\begin{equation} \label{PIneqs}
L_{k}(\phi_{\gamma}^{(n)}(x))\left. \right|_{x=0}=\sum_{j=0}^{2k} \frac{(-1)^{j+k}}{(2k)!} \binom{2k}{j} \gamma_{n+j}\gamma_{n+2k-j} \geq 0, \qquad 0 \leq k \leq d
\end{equation}
for all $n \geq N_{4}(d)$.
\end{Thm}
In Section 4 we give suggested values for $N_{4}(d)$ in Table \ref{PartitionLaguerre}  when $\gamma_{k}=p(k)$ is the integer partition function and we discuss further inequalities.

This paper is organized as follows.  In Section 2 we will give some basic properties of order $d$ multiplier sequences and Jensen polynomials.  In Section 3 we will prove Theorems \ref{SLP} and \ref{LThm}.  In Section 4 we will give some background discussing inequalities satisfied by multiplier sequences and discuss some iterated inequalities.  In particular, we make an ambitious conjecture about functions in the shifted Laguerre-P\'{o}lya class and give some explicit conjectures for inequalities on the partition numbers. 

\section*{Acknowledgements}
The author thanks Ken Ono and Larry Rolen for helpful discussions related to this project, Anneliese Koontz for translating \cite{PS}, and the referee for offering several helpful suggestions that improved this paper.

\section{Background and basic properties}
\subsection{Basic properties of order $d$ multiplier sequences}
This section contains some basic properties of order $d$ multiplier sequences.  The results and proofs of these lemmas mirror those in Section 1 of \cite{PS}.

\begin{Lemma}
Let $\{\gamma_{k} \}_{k \geq 0}$ be an order $d \geq 3$ multiplier sequence of type I and let $n \geq N(3)$.  Then any zero terms of $\{\gamma_{n+k}\}_{k \geq 0}$ are either at the beginning or are at the end of the sequence.  Furthermore, if $\{\gamma_{k} \}_{k \geq 0}$ is an order $d \geq 2$ multiplier sequence of type I and $n \geq N(2)$, then the nonzero terms of $\{\gamma_{n+k}\}_{k \geq 0}$ are all either the same sign or alternate signs.
\end{Lemma}
\begin{proof}
First, consider the two polynomials $f_{1}(x)=x^2 -1$ and $f_{2}(x)=x^2 +2x +1$, which have roots $x=\pm 1$ and $x=-1$ respectively.  If $\gamma= \{\gamma_{k}\}_{k \geq 0}$ is an order $d \geq 2$ multiplier sequence of type I and $n \geq N(2)$, then $\Gamma_{\{\gamma_{k+n}\}}(f_{1}(x))$ has roots $x= \pm \sqrt{\gamma_{n}/\gamma_{n+2}}$ and $\Gamma_{\{\gamma_{k+n}\}}(f_{2}(x))$ has roots $x =\frac{-\gamma_{n+1} \pm \sqrt{\gamma_{n +1}^2 - \gamma_{n} \gamma_{n+2}}}{\gamma_{n+2}}$.  These roots must all be real so we immediately find that for nonzero terms ${\rm{sign}}(\gamma_{n}) = {\rm{sign}}(\gamma_{n+2})$ and $\gamma_{n+1}^2 \geq \gamma_{n}\gamma_{n+2}$ for all $n \geq 0$. 

In \cite{S} Schur shows that if $h(x) = \sum_{k=0}^{d} c_{k}x^k$ is a real polynomial with real roots and $c_{0}, c_{d} \neq 0$, then $c_{r}^2 - \frac{r+1}{r} c_{r-1}c_{r+1} >0$ for $0<r<d$.  Let $g(x) = a_0+ a_1 x + a_2 x^2 + a_3 x^3$ be a real polynomial with all real roots and $a_i \neq 0$ for all $i$.  Let $\gamma= \{\gamma_{k}\}_{k \geq 0}$ be an order $d \geq 3$ multiplier sequence of type I and $n \geq N(3)$.  We may shift the sequence sufficiently in order to assume that $\gamma_n, \gamma_{n+3} \neq 0$.  Then $\Gamma_{\{\gamma_{n+k}\}}(g(x))$ has all real roots and by the above result of Schur $\gamma_{n+1}^2 a_{1}^2 -2 \gamma_n \gamma_{n+2} a_0 a_2 >0$ for all $n \geq N(3)$.  Because this inequality is strict, we can't have two consecutive terms of $\{\gamma_{n+k}\}_{k \geq 0}$ be equal to zero, unless a string of them are all zero which would make the $x^3$ or constant term of $\Gamma_{\{\gamma_{n+k}\}}(g(x))$ equal to zero.  Suppose that we have a single zero term $\gamma_{n+1}=0$ and $\gamma_{n}, \gamma_{n+2} \neq 0$.  Then the result of Schur implies that $\gamma_{n}$ and $\gamma_{n+2}$ must have opposite signs.  This contradicts the discussion above so a single zero term is not possible.

\end{proof}
The above proof can also be used to show the following lemma.
\begin{Lemma}
Let $\{\gamma_{k}\}_{k\geq 0}$ be an order $d \geq 3$ multiplier sequence of type II and let $n \geq N(3)$.  If $\gamma_{n}=\gamma_{n+1}=0$, then $\gamma_{j}=0$ either for all $j \geq n$ or all $N(d) \leq j \leq n$.  If $\gamma_{n}=0$ and $\gamma_{n-1} \gamma_{n+1} \neq 0$, then ${\rm{sign}}(\gamma_{n-1}) \neq {\rm{sign}}(\gamma_{n+1})$.
\end{Lemma}
The following lemma is clear.
\begin{Lemma}
If $\{\gamma_{k}\}_{k \geq 0}$ and $\{\lambda_{k}\}_{k \geq 0}$ are order $d_{1}$ and $d_2$ multiplier sequences respectively, then $\{\gamma_{k} \lambda_{k}\}_{k \geq 0}$ is an order $\min(d_1, d_2)$ multiplier sequence.  Further, if both $\{\gamma_{k}\}_{k \geq 0}$ and $\{\lambda_{k}\}_{k \geq 0}$ are type I sequences, then so is $\{\gamma_{k} \lambda_{k}\}_{k \geq 0}$.  Otherwise $\{\gamma_{k} \lambda_{k}\}_{k \geq 0}$ is of type II.
\end{Lemma}
The last basic property of order $d$ multiplier sequences follows from \cite{S} as well.
\begin{Lemma}
If $\{\gamma_{k,i}\}_{k \geq 0}$, $i=1, 2, 3, \dots$, are infinitely many order $d$ multiplier sequences of type I or II and for each $k$ there exists a $\gamma_{k}$ such that $\lim_{i \to \infty} \gamma_{k,i} = \gamma_{k}$, then $\{\gamma_{k}\}_{k \geq 0}$ is also an order $d$ multiplier sequence of type I or II.
\end{Lemma}

\subsection{Basic properties of Jensen polynomials and entire functions}
In this subsection we will compile some useful results on Jensen polynomials and related objects.  Recall that the Jensen polynomial of degree $d$ and shift $n$ for the sequence of real numbers $\{\gamma_{k}\}_{k \geq 0}$ is given by
\begin{equation*}
J_{\gamma}^{d,n}(x) = \sum_{k =0}^{d} \binom{d}{k} \gamma_{k+n} x^k.
\end{equation*}
The derivatives of Jensen polynomials are also Jensen polynomials.  In particular, we have
\begin{equation} \label{Jderiv}
\frac{d^m}{dx^m} J_{\gamma}^{d,n}(x) = \frac{d!}{(d-m)!} J_{\gamma}^{d-m, n+m}(x).
\end{equation}
This property allows us to prove the following useful result.
\begin{Lemma} \label{degree}
Let $\{\gamma_{k}\}_{k \geq 0}$ be a sequence of real numbers.  If $J_{\gamma}^{d,n}(x)$ has all real roots, then $J_{\gamma}^{m,n}(x)$ has all real roots for all $1 \leq m \leq d$.
\end{Lemma}
\begin{proof}
Assume $J_{\gamma}^{d,n}(x)$ has all real roots.  We can define the degree $d$ shift $n$ Appell polynomial associated to the sequence $\gamma$ by
\begin{equation*}
P_{\gamma}^{d,n}(x) := \frac{1}{d!} x^d J_{\gamma}^{d,n}(1/x),
\end{equation*}
so $P_{\gamma}^{d,n}(x)$ has all real roots if and only if $J_{\gamma}^{d,n}(x)$ does.  Using equation \eqref{Jderiv} for the Jensen polynomials above, a straightforward calculation shows
\begin{equation*}
\frac{d^m}{dx^m} P_{\gamma}^{d,n}(x) = P_{\gamma}^{d-m,n}(x).
\end{equation*}
Differentiation preserves real-rootedness so these polynomials also have real roots which completes the proof.
\end{proof}
Given a sequence $\{\gamma_{k}\}_{k \geq 0}$ we associate to it the formal power series $\phi_{\gamma}(x) := \sum_{k=0}^{\infty} \frac{\gamma_{k}}{k!} x^k$.  The Jensen polynomials can be seen to be a distinguished set of polynomials when studying $\phi_{\gamma}(x)$ due to the following result of Craven and Csordas \cite{CC}.
\begin{Lemma} \label{converge}
Let 
\begin{equation*}
\phi_{\gamma}(x) =\sum_{k=0}^{\infty} \frac{\gamma_{k}}{k!} x^k
\end{equation*}
be an arbitrary entire function. Then 
\begin{equation*}
\lim_{d \to \infty} J_{\gamma}^{d,n} \left(\frac{x}{d} \right) = \phi_{\gamma}^{(n)}(x)
\end{equation*}
uniformly on compact subsets of $\C$.
\end{Lemma}
Our final preparatory lemma is a well-known result due to Schur and Szeg\"{o}.  See for example \cite{L}.
\begin{Lemma}[The Schur-Szeg\"{o} composition theorem] \label{SS}
Let 
\begin{equation*}
f_1(x) = \sum_{k=0}^{d} \binom{d}{k} a_{k} x^k, \qquad f_2(x) = \sum_{k=0}^{d} \binom{d}{k} b_{k}x^k
\end{equation*}
be two real polynomials.  If $f_1$ and $f_2$ have all real roots and the roots of $f_2$ all have the same sign, then the polynomial
\begin{equation*}
g(x) = \sum_{k=0}^{d} \binom{d}{k} a_k b_k x^k
\end{equation*}
also has all real roots.
\end{Lemma}

\section{Proofs of theorems}
\begin{proof}[Proof of Theorem \ref{SLP}]
Throughout this proof we will assume without loss of generality that all sequences and Taylor coefficients are eventually non-negative. We will begin by showing that (1) and (2) are equivalent.  Assume that there is an $N_{1}(d)$ such that $\{\gamma_{k+n}\}_{k \geq 0}$ is an order $d$ multiplier sequence of type I for all $n \geq N_{1}(d)$.  Then apply $\Gamma_{\{\gamma_{k+n}\}}$ to $(1+x)^d$ to imply (2).  On the other hand, assume that there is an $N_{2}(d)$ such that $J_{\gamma}^{d,n}(x)=\sum_{k=0}^{d} \binom{d}{k} \gamma_{k + n} x^k$ has all real zeros of the same sign for all $n \geq N_{2}(d)$.  From Lemma \ref{degree} this also implies that $J_{\gamma}^{m,n}(x)=\sum_{k=0}^{m} \binom{m}{k} \gamma_{k+n}x^k$ has all real roots of the same sign for $m \leq d$.  Let $f(x) = \sum_{k=0}^{m}a_{k}x^k = \sum_{k=0}^{m} \binom{m}{k} \left[ \binom{m}{k}^{-1} a_k \right]x^k$ be a polynomial with all real zeros and $m \leq d$.  Then the Schur-Szeg\"o composition theorem implies that $\sum_{k=0}^{m} a_{k} \gamma_{k+n} x^k$ has all real roots for $m \leq d$.

We will now show that (2) implies (3).  Assume that for each $d \in \NN$, there exists an $N_{2}(d)$ such that $J_{\gamma}^{d,n}(x)$ has all real roots for $n \geq N_{2}(d)$.  From Lemma \ref{converge} we know that $J_{\gamma}^{d,n} \left(\frac{x}{d}\right)$ converges uniformly to $\phi_{\gamma}^{(n)}(x)$.  In particular, $J_{\gamma}^{d,0}\left(\frac{x}{d}\right)$ converges uniformly to $\phi_{\gamma}(x)$.  Equation \eqref{Jderiv} then implies that $\frac{d^{n}}{dx^{n}} J_{\gamma}^{d+n,0}\left(\frac{x}{d}\right)$ has all real roots for any $n \geq N_{2}(d)$.  Note that because $J_{\gamma}^{2,n}(x)$ has all real roots for $n$ large enough, we also have $\gamma_{n}^2 - \gamma_{n-1} \gamma_{n+1} \geq 0$ for $n$ large enough which is sufficient to show $\phi_{\gamma}(x)$ is entire.  We now need to show that $\phi_{\gamma}(x)$ has order at most $1$.  If each $\gamma_{n}=0$, then the conclusion holds so suppose $\gamma_{r}$ is the first nonzero term of the sequence.  Let $m>r$, then $J_{\gamma}^{m,n}(x)$ has all real roots for all $n \geq N_{2}(m)$ so we have
\begin{equation*}
J_{\gamma}^{m,n}(x) = \binom{m}{r} \gamma_{r+n}x^r \prod_{k=1}^{m-r}(1+r_{m,k}x) 
\end{equation*}
where the $r_{m,k}$ are all real and positive.  We then have
\begin{equation*}
\binom{m}{r+1} \gamma_{r+n+1} = \binom{m}{r} \gamma_{r+n} \sum_{k=1}^{m-r} r_{m,k}
\end{equation*}
and
\begin{equation*}
\gamma_{m+n} = \binom{m}{r} \gamma_{r+n} \prod_{k=1}^{m-r} r_{m,k}.
\end{equation*}
We follow \cite{PS} and use the AM-GM inequality to obtain
\begin{equation*}
\gamma_{m+n} \leq \binom{m}{r} \gamma_{r+n} \left[ \frac{\gamma_{r+n+1}}{(r+1)\gamma_{r+n}} \right]^{m-r},
\end{equation*}
or equivalently
\begin{equation} \label{order}
\frac{\gamma_{m+n}}{(m+n)!} \leq \frac{1}{(m+n)!} \binom{m}{r} \gamma_{r+n} \left[ \frac{\gamma_{r+n+1}}{(r+1)\gamma_{r+n}} \right]^{m-r}.
\end{equation}
Due to the eventually log-concavity of the sequence we have that for $k$ big enough $\frac{\gamma_{k+1}}{(k+1)!} \leq \frac{\gamma_{k}}{k!}$ so we can take a limit $m \to \infty$ in equation \eqref{order} with $n$ and $r$ fixed which shows that $\phi_{\gamma}(x)$ has order at most $1$.  Lastly, we will show that $\sigma \geq 0$ where  
\begin{equation} \label{genus}
\phi_{\gamma}(x) = \sum_{k=r}^{\infty} \frac{\gamma_{k}}{k!} x^k = \frac{\gamma_{r}}{r!} x^r e^{\sigma x} \prod(1+ \rho_k x).
\end{equation}
As above, we can write
\begin{equation*}
J_{\gamma}^{d,0}(x) = \binom{d}{r} \gamma_r x^r \prod_{k=1}^{d-r}(1+\rho_{d,k}x),
\end{equation*}
where the $\rho_{d,k}$ now may not necessarily be real.  From this we find
\begin{equation*}
\sum_{k =1}^{d-r} \frac{\rho_{d,k}}{d} = \frac{d-r}{d} \frac{\gamma_{r+1}}{(r+1)\gamma_{r}} \leq \frac{\gamma_{r+1}}{(r+1)\gamma_r}
\end{equation*}
and because $\lim_{d \to \infty} J_{\gamma}^{d,0} \left(\frac{x}{d}\right) = \phi_{\gamma}(x)$ we obtain
\begin{equation*}
0 \leq \sum_{k=1}^{\infty} \rho_k \leq \frac{\gamma_{r+1}}{(r+1)\gamma_r}.
\end{equation*}
Because $\phi_{\gamma}(x)$ is real any complex $\rho_{d,k}$ or $\rho_k$ show up in complex conjugate pairs so the sums above are real.  From equation \eqref{genus} we obtain $\frac{\gamma_{r+1}}{(r+1)!} = \frac{\gamma_r}{r!} \left( \sigma + \sum_{k} \rho_k \right)$ or equivalently $\sigma = \frac{\gamma_{r+1}}{(r+1)\gamma_r} - \sum_{k} \rho_k \geq 0$.  Note that the exact same proof shows $\sigma = \frac{\gamma_{r+1}}{(r+1) \gamma_r} \geq 0$ when
\begin{equation*}
\phi_{\gamma}(x) = \frac{\gamma_r}{r!} x^r e^{\sigma x} \prod_{k=1}^{\infty} \left( 1 + \rho_k x \right)e^{-\rho_k x}.
\end{equation*}

Finally, we will show that (3) implies (1).  Assume that $\phi_{\gamma}(x) = \sum_{k=0}^{\infty} \frac{\gamma_k}{k!}x^k$ is a real entire function which is the uniform limit of a sequence of polynomials $\{\phi_{m}\}_{m \geq 0}$ with degrees $d_{m}$ such that for each $d \in \NN$, there exists an $N_{3}(d)$ such that $\phi_{n}^{(d_n -d)}(x)$ has all real roots  for any $n \geq N_{3}(d)$.  Let 
\begin{equation*}
\phi_{m}(x) = \gamma_{0,m} + \gamma_{1,m}x + \frac{\gamma_{2,m}}{2!}x^2 + \cdots.
\end{equation*}
Then it is known that $\lim_{m \to \infty} \gamma_{k,m} =\gamma_{k}$ and further $\{\gamma_{k,m}\}_{m \geq 0}$ forms a Cauchy sequence.  Let $m$ be sufficiently large.  Then
\begin{equation*}
\phi_{m}^{(d_m-d)}(x) = \gamma_{d_m-d,m} + \gamma_{d_m -d+1,m}x + \frac{\gamma_{d_m-d+2}}{2!}x^2 + \cdots + \frac{\gamma_{d_m}}{d!}x^d
\end{equation*}
has degree $d$ and all real non-positive roots.  Let $f(x) = \sum_{k=0}^{d} a_{k}x^k$ have all real roots.  Then it is known that
\begin{equation*}
\sum_{k=0}^{d} k! a_{k} \frac{\gamma_{d_m -d+k,m}}{k!} x^k = \sum_{k=0}^{d} a_{k} \gamma_{d_m -d+k, m} x^k
\end{equation*}
also has real roots which shows that $\{\gamma_{d_m -d +k,m}\}_{k=0}^{d}$ is an order $d$ multiplier sequence of type I.  The uniform convergence implies that for $m$ fixed and sufficiently large $\{\gamma_{d_m -m +k, t}\}_{k=0}^{d}$ is also an order $d$ multiplier sequence of type $1$ for all $t \geq m$.  Taking a limit gives the result for $\{\gamma_{d_m -d +k}\}_{k=0}^{d}$.  We also know that $d_m \to \infty$ as $m \to \infty$ so the same argument can be repeated with $m$ shifted to show the later entries in the sequence also form an order $d$ multiplier sequence of type I.  Together, this shows that for each $d \in \NN$ there exists an $N(d)$ such that $\{\gamma_{k+n}\}_{k \geq 0}$ for all $n \geq N(d)$ which completes the proof.

\end{proof}
\begin{proof}[Proof of Theorem \ref{LThm}]
Suppose that $\phi(x) \in \mathcal{SL\text{-}P}$ and let $\{\phi_{k}\}_{k \geq 0}$ be a sequence of polynomials converging to $\phi$ uniformly as in the definition of the shifted Laguerre-P\'{o}lya class such that the derivatives of the polynomials converge uniformly to the same derivative of $\phi$ (for example, one could take the associated Jensen polynomials).  Then for each $d \in \NN$, there exists an $N(2d)$ such that $\phi_{n}^{(d_n -2d)}(x)$ has all real roots for every $n \geq N(2d)$.  Therefore $L_{k}\left( \phi_{n}^{(d_n-2d)}(x) \right) \geq 0$ for any $x \in \R$.    Uniform convergence also implies that there exist $N_{t}$ such that for $n \geq N_{t}$, we have $|\phi_{n}^{(t)}(x) - \phi^{(t)}(x)|<\epsilon$ for any $x \in \R$ and positive $\epsilon$.  Thus if we take $n$ to be larger than the maximum of $N_{d_n -2d +j}$ for $0 \leq j \leq 2d$, then due to uniform convergence, we also have $L_{k} \left( \phi_{m}^{(d_n -d)}(x) \right) \geq 0$ for $m \geq n$ and $0 \leq k \leq d$.  We can then take a limit with respect to $m$ to complete the proof.
\end{proof}

\section{Iterated inequalities}
\subsection{Iterated Tur\'{a}n inequalities}
The language ``higher Tur\'{a}n inequality" has been used historically to describe the inequality 
\begin{equation} \label{T3}
4(\gamma_{i+1}^2 - \gamma_{i}\gamma_{i+2})(\gamma_{i+2}^2 - \gamma_{i+1} \gamma_{i+3}) - (\gamma_{i+1} \gamma_{i+2} - \gamma_{i}\gamma_{i+3})^2 \geq 0.
\end{equation}
However, recently it has been referred to as the third order Tur\'{a}n inequality.  The language has changed for the following reasons.  Let $f(x) = b_d x^{d} + b_{d-1}x^{d-1} + \cdots + b_1 x + b_0$ be a polynomial with real coefficients and let $r_1, \dots, r_d$ denote the roots of $f$.  Let $S_{m} = \sum_{i=1}^{d} r_{i}^{m}$ be the power sums of the roots of $f$ and define the Hankel matrix of $f$ by
\begin{equation*}
M(f) := \begin{pmatrix} S_{0} & S_{1} & S_{2} & \cdots & S_{d-1} \\
S_{1} & S_{2} & S_{3} & \cdots & S_{d} \\ 
S_2 & S_3 & S_4 & \cdots & S_{d+1} \\
\vdots & \vdots & \vdots & \ddots & \vdots \\
S_{d-1} & S_d & S_{d+1} & \cdots & S_{2d-2} \end{pmatrix}.
\end{equation*}
Hermite's theorem states that all of the roots of $f(x)$ are real if and only if $M(f)$ is positive definite.  It is well known that a matrix is positive definite if and only if all of it's principle minors are positive.  This then makes the hyperbolicity of $f(x)$ equivalent to the set of inequalities
\begin{align*}
D_1 &:= S_0 =d >0, \\
D_2 &:= \begin{vmatrix} S_0 & S_1 \\ S_1 & S_2 \end{vmatrix}>0, \\
&\vdots \\
D_{d} &:= \begin{vmatrix} S_{0} & S_{1} & S_{2} & \cdots & S_{d-1} \\
S_{1} & S_{2} & S_{3} & \cdots & S_{d} \\ 
S_2 & S_3 & S_4 & \cdots & S_{d+1} \\
\vdots & \vdots & \vdots & \ddots & \vdots \\
S_{d-1} & S_d & S_{d+1} & \cdots & S_{2d-2} \end{vmatrix}>0.
\end{align*}
One can use the Newton identities to write each $S_m$ in terms of the coefficients of $f(x)$ which leads to the roots of $f$ being real being equivalent to some inequalities on the coefficients of $f$.  For example, $D_2>0$ is equivalent to $(d-1)b_{d-1}^2 - 2db_d b_{d-2}>0$.
Using the argument above, the Jensen polynomial $J_{\gamma}^{d,n}(x)$ having all real roots for $n \geq N$ is equivalent to the sequence $\{\gamma_{k+n} \}_{k\geq 0}$ satisfying $d$ separate inequalities coming from $D_{j}>0$ for $1 \leq j \leq d$ for $n \geq N$.  We call each one of these the order $j$ Tur\'{a}n inequality.  For example, when $j=2$ we recover log-concavity.  We will adopt the following notation.  The order $j$ T\'uran inequality for a given sequence $\{\gamma_{k}\}_{k \geq 0}$ will be denoted by $T_{j}(\gamma_{i})$ or $T_{j}(i)$ when the sequence is clear.  For example, the statement that a sequence $\gamma$ is log-concave can be written as $T_{2}(i) = \gamma_{i}^2 - \gamma_{i-1} \gamma_{i+1} >0$.  For completeness we will define the order $1$ T\'uran inequality as $T_{1}(i) = \gamma_{i} - \gamma_{i-1} >0$ which is just the statement that the sequence is increasing.  We will also use the usual notation to compose the inequalities.  Therefore we have $T_{j}^{(2)}(i) = T_{j}(T_{j}(i))$ and $T_{j}^{(k)}(i) = T_{j}(T_{j}^{(k-1)}(i))$.  For example,
\begin{equation*}
T_{2}^{(2)}(i) = (\gamma_{i}^2 - \gamma_{i-1}\gamma_{i+1})^2 - (\gamma_{i-1}^2 - \gamma_{i-2} \gamma_{i})(\gamma_{i+1}^2 - \gamma_{i} \gamma_{i+2}),
\end{equation*}
\begin{equation*}
T_{1}(T_{2}(i)) = (\gamma_{i}^2 - \gamma_{i-1}\gamma_{i+1}) - (\gamma_{i-1}^2 - \gamma_{i-2}\gamma_{i}),
\end{equation*}
and 
\begin{equation*}
T_{1}^{(k)}(i) = \sum_{m=0}^{k} \binom{k}{m} (-1)^{k-m} \gamma_{i+m}.
\end{equation*}
We also define the operators $\mathcal{T}_{j}^{(k)}(\{\gamma_{i}\}_{i \geq 0}) := \{T_{j}^{(k)}(i)\}_{i \geq 0}$.  We say a sequence $\{\gamma_{i}\}_{i \geq 0}$ is {\bf{$k$-log-concave}} if $T_{2}^{(k)}(i) \geq 0$ for all $i \geq 0$ and say a sequence is {\bf{infinitely-log-concave}} if it is $k$-log-concave for all $k \geq 1$.  Clearly $1$-log-concavity is usual log-concavity.  A multiplier sequence automatically satisfies all of the higher Tur\'{a}n inequalities so it is log-concave, but it is also conjectured that it is infinitely-log-concave.  In \cite{CC} Craven and Csordas show that multiplier sequences are $3$-log-concave and a similar proof should hold for shifted multiplier sequences, but the author has not found much work on iterating higher order Tur\'{a}n inequalities.

The analogous conjecture for shifted multiplier sequences is that for any $k$ the sequence is eventually $k$-log-concave.  We also expect the same to be true for iterations of higher Tur\'{a}n inequalities. 
\begin{Conj}
Assume the notation above and let $\{\gamma_{i}\}_{i \geq 0}$ be a shifted multiplier sequence that eventually is increasing.  Then for each $j, k \in \NN$ there exists an $N_{j}^{(k)}$ such that $T_{j}^{(k)}(n) >0$ for all $n \geq N_{j}^{(k)}$.
\end{Conj} 

We provide evidence for this conjecture with the sequence given by the integer partition numbers $p(n)$.  In Table \ref{PartitionTuran} below we give the minimum conjectured value for $N_{j}^{(k)}$ which we speculate may have size approximately $\frac{6}{\pi^2} (jk)^2 \log(jk)^2$.
\begin{center}
\begin{table}[H]
\begin{tabular}{ || c | c | c | c | c ||}
\hline
 $j \backslash k$ & 1 & 2 & 3 & 4   \\ 
 \hline
 $1$ & 2 & 8 & 26 & 68  \\
 \hline
 $2$ & 26 & 222 & 640 & 1292   \\
 \hline   
 $3$ & 94 & 522 & 1232 & 2094  \\
 \hline
 $4$ & 206 & 991 & 2040 & 3005 \\
 \hline
\end{tabular}
 \caption{Minimum conjectured $N_{j}^{(k)}$ such that $T_{j}^{(k)}(p(n)) > 0$ for all $n \geq N_{j}^{(k)}$.}
\label{PartitionTuran}
\end{table}
\end{center}

We note that the values of $N_{j}^{(1)}$ in Table \ref{PartitionTuran} (and $N_{5}^{(1)} =381$) have been proved over several papers \cite{N, DP, LW}.  Further, in \cite{G} Gupta showed $N_{1}^{(k)}$ exists for all $k$ and computed $N_{1}^{(k)}$ explicitly for $k \leq 10$ and Odlyzko \cite{O} showed that $N_{1}^{(k)} \sim \frac{6}{\pi^2} k^{2} \log(k)^2$.   Hou and Zhang have proved that $N_{2}^{(k)}$ exists for the partition function for every $k$ in \cite{HZ} and confirmed the table for $k \leq 3$.  We in fact believe much more might be true than stated above.  Any composition of the Tur\'{a}n inequalities seems to eventually hold.  This leads to the following ambitious conjecture.
\begin{Conj} \label{TC}
Assume the notation above and let $\{\gamma_{i}\}_{i \geq 0}$ be a shifted multiplier sequence that eventually is increasing.  Then for each $j \in \NN$, $\mathcal{T}_{j}(\{\gamma_{i}\}_{i \geq 0})$ is again a shifted multiplier sequence that is eventually increasing.
\end{Conj}

\subsection{Iterated Laguerre inequalities}
In \cite{CC} Craven and Csordas also investigate iterating the classical Laguerre inequality $L_{1}(\phi(x))$ for $\phi(x) \in \mathcal{L\text{-}P}$ with non-negative Taylor coefficients.  In particular, they show that for such a function $L_{1}^{(2)}(\phi^{(k)}(x)) \geq 0$ for all $k \geq 2$ and $x \in \R$, where as above $L_{j}^{(2)}(\phi(x)) = L_{j}(L_{j}(\phi(x)))$ and $L_{j}^{(k)}(\phi(x)) = L_{j}(L_{j}^{(k-1)}(\phi(x)))$.   We make the following conjecture analogous to Conjecture \ref{TC}.
\begin{Conj}
Let $\phi(x) \in \mathcal{SL\text{-}P}^{+}$.  Then for any $j \in \NN$, $L_{j}(\phi(x)) \in \mathcal{SL\text{-}P}^{+}$.
\end{Conj}
This conjecture seems to have been investigated less than the Tur\'{a}n case for the partition function.  For example, at $x=0$ we have $L_{1}(\phi_{p}^{(n)}(0)) = p(n+1)^2 - p(n)p(n+2) \geq 0$ which is log-concavity, but the inequalities
\begin{equation*}
L_{k}(\phi_{p}^{(n)}(0)) = \sum_{j=0}^{2k} \frac{(-1)^{k+j}}{(2k)!} \binom{2k}{j} p(n+j) p(n+2k-j) \geq 0, \quad n \geq N_k
\end{equation*}
proved in Theorem \ref{PL} did not seem to be known for $k>1$ before this work.  Below in Table \ref{PartitionLaguerre}, we give the minimal conjectured $N_{j}^{(1)}$ such that $L_{j}(\phi_{p}^{(n)}(0)) \geq 0$ for all $n \geq N_{j}^{(1)}$.

\begin{center}
\begin{table}[H]
 \begin{tabular}{||c|c|c|c|c|c|c|c|c|c|c||} 
 \hline
  $j$  & 1 & 2 & 3 & 4 & 5 & 6 & 7 & 8 & 9 & 10 \\
 \hline
 $N_{j}^{(1)}$ & 25 & 184 & 531 & 1102 & 1923 & 3014 & 4391 & 6070 & 8063 & 10382 \\ 
 [1ex] 
 \hline
\end{tabular}
 \caption{Minimum conjectured $N_{j}^{(1)}$ such that $L_{j}(\phi_{p}^{(n)}(0)) \geq 0$ for all $n \geq N_{j}^{(1)}$.}
\label{PartitionLaguerre}
\end{table}
\end{center} 
\begin{rmk}
Shortly after the preparation of this paper the author was notified that the value of $N_{2}^{(1)}=184$ was confirmed by Wang and Yang in \cite{WY}.
\end{rmk}
Note that at even though $L_{1}(\phi_{\gamma}(0)) = T_{2}(\gamma_{i+1})$, their iterations are not equal.  For example
\begin{equation*}
T_{2}^{(2)}(\gamma_{i+1}) =(\gamma_{i+1}^2 - \gamma_{i} \gamma_{i+2})^2 - (\gamma_{i}^2 - \gamma_{i-1} \gamma_{i+1})(\gamma_{i+2}^2 -\gamma_{i+1}\gamma_{i+3})
\end{equation*}
while
\begin{equation*}
L_{1}^{(2)}(\phi_{\gamma}(0)) = (\gamma_{i+1}\gamma_{i+2} - \gamma_{i}\gamma_{i+3})^2 -(\gamma_{i+1}^2 -\gamma_{i}\gamma_{i+2})(\gamma_{i+2}^2 - \gamma_{i}\gamma_{i+4}).
\end{equation*}
One can see that $L_{1}^{(2)}, T_{2}^{(2)}$, and $T_{3}$ in equation \eqref{T3} all have the same degree, but are not equal.


\begin{thebibliography}{45}

\bibitem{CJW}
W.Y.C.~Chen, D.X.Q.~Jia, and L.X.W.~Wang, Higher Order Tur\'{a}n Inequalities for the Partition
Function, Trans. Amer. Math. Soc. 372 (2019), 2143-2165.

\bibitem{CC}
T.~Craven and G.~Csordas. {\it Jensen polynomials and the Tur\'{a}n and Laguerre inequalities}, Pacific J. Math., {\bf 136}, no. 2 (1989), 241--260.

\bibitem{CCS}
T.~Craven, G.~Csordas, and W.~Smith. {\it The zeros of derivatives of entire functions and the P\'{o}lya-Wiman Conjecture}, Ann. of Math. {\bf 125}, no. 2 (1987), 405--431.

\bibitem{CV1}
G.~Csordas and R.S.~Varga. {\it Necessary and sufficient conditions and the Riemann hypothesis}, Adv. in Appl. Math., {\bf 11}, no. 3 (1990), 328--357.

\bibitem{CV2}
G.~Csordas and A.~Vishnyakova. {\it The generalized Laguerre inequalities and functions in the Laguerre-P\'{o}lya class}, Cent. Eur. J. Math., {\bf 11}, no. 9 (2013), 1643--1650.

\bibitem{DM}
M.~Locus Dawsey and R.~Masri. {\it Effective bounds for the Andrews spt-function}, Forum Mathematicum, {\bf 31} no. 3 (2019), 743--767.

\bibitem{DP}
S.~DeSalvo and I.~Pak. {\it Log-concavity of the partition function}, Ramanujan J. {\bf 38}, (2015), 61--73.

\bibitem{GORTTW}
M.~Griffin, K.~Ono, L.~Rolen, J.~Thorner, Z.~Tripp, and I.~Wagner. {\it Jensen polynomials for thr Riemann Xi function}, https://arxiv.org/abs/1910.01227.

\bibitem{GORZ}
M.~Griffin, K.~Ono, L.~Rolen, and D.~Zagier. {\it Jensen polynomials for the Riemann zeta function and other sequences}, Proc. Natl. Acad. Sci., USA {\bf 116}, no. 23 (2019), 11103-11110.

\bibitem{G}
H.~Gupta. {\it Finite differences of the partition function}, Math. Comp., {\bf 32}, no. 144 (1978), 1241--1243.

\bibitem{HN}
B.~Heim and M.~Neuhauser. {\it Log-concavity of recursively defined polynomials}, J. Integer Seq. {\bf 22}, 12, (2019).

\bibitem{HN2}
B.~Heim and M.~Neuhauser. {\it Log-concavity of infinite product generating functions}, arXiv:2202.00627.

\bibitem{HZ}
Q.-H.~Hou and Z.-R.~Zhang. {\it $r$-log-concavity of partition functions}, Ramanujan J., {\bf 48}, (2019), 117--129.

\bibitem{JW}
D.~Jia and L.~Wang. {\it Determinantal inequalities for the partition function}, Proc. Royal Soc. Edinburgh Sec. A, {\bf 150}, no. 3 (2020), 1451--1466.

\bibitem{K}
Y-O.~Kim. {\it A proof of the P\'{o}lya-Wiman Conjecture}, Proc. Natl. Acad. Sci., USA {\bf 109}, no. 4 (1990), 1045--1052.

\bibitem{L}
B.Ja.~Levin. {\it Distribution of Zeros of Entire Functions}, Amer. Math. Soc. Transl., Providence, R.I., 1964.

\bibitem{LW}
H.~Larson and I.~Wagner. {\it Hyperbolicity of the partition Jensen polynomials}, Res. number theory, {\bf 5}, (2019), 19.

\bibitem{N}
J.L.~Nicolas. {\it Sur les entiers $N$ pour lesquels il y a beaucoup de groupes ab\'{e}liens d'ordre $N$}, Ann. Inst. Fourier, {\bf 28}, no. 4 (1978), 1--16.

\bibitem{O}
A.M.~Odlyzko. {\it Differences of the partition function}, Acta Arith., {\bf 49}, no. 3 (1988), 237--254.

\bibitem{OPR}
K.~Ono, S.~Pujahari, and L.~Rolen. {\it Tur\'{a}n inequalities for the plane partition function}, arXiv:2201.01352.

\bibitem{P}
G.~P\'{o}lya. {\it \"{U}ber die algebraisch-funktionentheoretischen Untersuchungen von J.L.W.V. Jensen}, Kgl Danske Vid Sel Math-Fys Medd, {\bf 7}, (1927), 3--33.

\bibitem{PS}
G.~P\'{o}lya and J.~Schur. {\it \"{U}ber zwei Arten Faktorenfolgen in der Theorie der algebraischen Gleichungen}, Journal f\"{u}r die reine und angewandte Mathematik, 144 (1914), 89--113.

\bibitem{S}
J.~Schur. {\it Zwei S\"{a}tze \"{u}ber algebraische Gleichungen mit lauter reellen Wurzeln}, Journal f\"{u}r die reine und angewandte Mathematik, 144 (1914), 75--88.

\bibitem{RS1}
R.~Stanley. {\it Log-concave and unimodal sequences in algebra, combinatorics, and geometry}, Graph theory and its applications: East and West (Jinan, 1986), 500-535, Ann. New York Acad. Sci., 576, New York Acad. Sci., New York, 1989.

\bibitem{W}
I.~Wagner. {\it The Jensen-P\'{o}lya program for various L-functions}, Forum Mathematicum, {\bf 32}, no. 2 (2020), 525--539.

\bibitem{WY}
L.X.W.~Wang and E.Y.Y.~Yang. {\it Laguerre inequalities of discrete sequences}, preprint.


\end{thebibliography}
\end{document}